\documentclass[11pt]{article}

\usepackage[T1]{fontenc}
\usepackage{lmodern}
\usepackage{amsmath,amssymb,amsthm,mathtools}
\usepackage{microtype}
\usepackage[hidelinks]{hyperref}
\usepackage{enumitem}
\usepackage[a4paper,margin=30mm]{geometry}

\newtheorem{theorem}{Theorem}[section]
\newtheorem{proposition}[theorem]{Proposition}
\newtheorem{lemma}[theorem]{Lemma}

\theoremstyle{definition}
\newtheorem{definition}[theorem]{Definition}
\newtheorem{example}[theorem]{Example}
\theoremstyle{remark}
\newtheorem{remark}[theorem]{Remark}

\title{Sharp $p$-adic Extrema and Least Extremal Rows for Restricted Binomial GCDs}
\author{John Fairfax-Ball}
\date{}

\begin{document}
\maketitle

\begin{abstract}
For integers $m\ge 2$ and rows $N>m$ divisible by $m$, consider the restricted binomial greatest common divisor
\[
  G(N;m)=\gcd\left\{\binom Nk:0<k<N,\ m\mid k\right\}.
\]
Fix a prime $p$ with $p\nmid m$, and let $r_p(m)$ be the least positive integer $r$ such that $m<p^r$. We prove that the largest possible value of $v_p(G(N;m))$, as $N$ ranges over all admissible rows, is exactly $r_p(m)$, and we give a constructive equality row. Attainment is established before the least extremal row $T_p(m)$ is defined. For the special family $m=p^a+1$ with $a\ge2$, we determine that least row exactly:
\[
  T_p(p^a+1)=p^{3a}+1.
\]
The proof combines Kummer's carry theorem with an explicit leading-digit witness for the universal upper bound, a multiplicative-order construction for equality, and a separate strict-minimality argument below $p^{3a}+1$ with a fallback witness for the unique worst leading-digit pattern. The selected-GCD family itself is known in the literature; the relation to earlier results and the limits of the documented literature search are stated explicitly.
\end{abstract}

\section{Introduction}

Divisibility patterns in Pascal's triangle are naturally encoded by $p$-adic valuations of binomial coefficients. Kummer's theorem translates these valuations into carry counts in base $p$, and this viewpoint has long been used to study greatest common divisors of sets of binomial coefficients. We consider a fixed divisibility restriction on the lower index. For $m\ge2$ and an admissible row $N>m$ with $m\mid N$, set
\[
  G(N;m)=\gcd\left\{\binom Nk:0<k<N,\ m\mid k\right\}.
\]
Writing $N=mn$, this is exactly Wu's family
\[
  g(m,n)=\gcd\left\{\binom{mn}{mk}:1\le k<n\right\};
\]
thus the family itself is prior art \cite{Wu2026}. McTague studied the same multiple-index restriction earlier and, among other things, determined the prime valuation under the congruence condition $p\equiv1\pmod m$ \cite{McTague2017}.

The question addressed here is extremal in the row variable. For a prime $p$ not dividing $m$, define
\[
  r_p(m)=\min\{r\ge1:m<p^r\}.
\]
Our first theorem determines the sharp uniform upper bound for $v_p(G(N;m))$ over all admissible rows and proves that the bound is attained. The equality construction is explicit in terms of the leading base-$p$ digit of $m$ and a suitable exponent in the multiplicative order of $p$ modulo $m$.

Having first proved that equality rows exist, we may define the least extremal row $T_p(m)$. The second main theorem determines $T_p(m)$ for the family $m=p^a+1$, $a\ge2$. If
\[
  Q=p^a,\qquad q_0=Q^2-Q+1,
\]
then
\[
  (Q+1)q_0=Q^3+1,
\]
and we prove that $q_0$ is the least multiplier attaining the maximum $a+1$. The proof has two genuinely different components: exact attainment at $Q^3+1$, and a strict witness construction at every smaller multiplier. In the latter argument, the ordinary leading-digit split works except in one explicit carry pattern; there a fixed fallback multiplier $p^{a-1}$ supplies the required strict bound.

The paper is organized as follows. Section~\ref{sec:setup} fixes notation and records the residue form of Kummer's theorem. Sections~\ref{sec:upper} and \ref{sec:attainA} prove the universal bound and constructive attainment for the first theorem. Section~\ref{sec:Tdef} defines the least extremal row only after nonemptiness has been established. Sections~\ref{sec:Battain} and \ref{sec:Bminimal} prove attainment and strict minimality for $m=p^a+1$. Section~\ref{sec:examples} gives examples, including $T_2(5)=65$. Section~\ref{sec:related} describes prior art and the documented literature-search caveat. Section~\ref{sec:formal} records the independent Lean formalization and public Palomar registration.

\section{Notation and Kummer's theorem}\label{sec:setup}

Throughout, $p$ denotes a prime. For a nonzero positive integer $n$, $v_p(n)$ is the exponent of $p$ in $n$. All GCDs below are GCDs of nonempty finite sets of positive integers.

\begin{definition}
Let $m\ge2$. A row $N$ is \emph{admissible for $m$} if
\[
  N>m\qquad\text{and}\qquad m\mid N.
\]
For such a row define
\[
  G(N;m)=\gcd\left\{\binom Nk:0<k<N,\ m\mid k\right\}.
\]
For a prime $p$ put
\[
  r_p(m)=\min\{r\ge1:m<p^r\}.
\]
\end{definition}

If $p\nmid m$, then $r_p(m)=\lceil\log_p m\rceil$, but the integer-power definition avoids any dependence on real logarithms. Since $p$ is prime and $m\ge2$, the selected set in an admissible row is nonempty: if $N=mq$ with $q\ge2$, then $k=m$ is selected.

The valuation of a GCD is the minimum of the valuations of its entries, so
\begin{equation}\label{eq:gcdmin}
  v_p(G(N;m))
  =\min_{\substack{0<k<N\\m\mid k}}v_p\binom Nk.
\end{equation}
The fundamental input is Kummer's theorem \cite{Kummer1852,Granville1997}.

\begin{lemma}[Kummer, residue/borrow form]\label{lem:kummer-residue}
For $0\le k\le N$,
\[
  v_p\binom Nk
  =\#\left\{i\ge1:(k\bmod p^i)>(N\bmod p^i)\right\}.
\]
Equivalently, if $N=x+y$ and $k=x$, then $v_p\binom Nk$ is the number of carries in the base-$p$ addition $x+y$.
\end{lemma}

\begin{proof}
Subtract $k$ from $N$ in base $p$. A borrow is present across the $p^i$ boundary exactly when the integer represented by the lower $i$ digits of $k$ is larger than that represented by the lower $i$ digits of $N$, namely when $k\bmod p^i>N\bmod p^i$. Kummer's theorem identifies the number of borrows in the subtraction with the number of carries in $k+(N-k)$ and with $v_p\binom Nk$.
\end{proof}

\section{The sharp universal upper bound}\label{sec:upper}

We first prove the half of the main extremal theorem that holds row by row.

\begin{proposition}[Universal upper bound]\label{prop:upper}
Let $p$ be prime and let $m\ge2$ with $p\nmid m$. Put $r=r_p(m)$. Then every admissible row $N$ satisfies
\[
  v_p(G(N;m))\le r.
\]
\end{proposition}

\begin{proof}
Write $N=mq$ with $q\ge2$. Split $q$ at its leading base-$p$ digit:
\[
  q=\lambda p^t+b,
  \qquad 1\le\lambda\le p-1,
  \qquad 0\le b<p^t.
\]
We construct one selected lower index whose binomial valuation is at most $r$.

Suppose first that $b>0$. Take $k=mb$ and write
\[
  mb=u+p^tX,
  \qquad 0\le u<p^t,
  \qquad X=\left\lfloor\frac{mb}{p^t}\right\rfloor.
\]
Since $b<p^t$, we have $0\le X<m<p^r$. The complementary summand is
\[
  N-k=m\lambda p^t.
\]
Below digit $t$ this summand is zero, so no carry occurs below $t$ and no carry enters digit $t$. After discarding the common shift by $p^t$, the remaining carry pattern is exactly that of
\[
  X+\lambda m.
\]
Moreover
\[
  X+\lambda m<m+(p-1)m=pm<p^{r+1}.
\]
Thus no carry can leave digit $r$, and there are at most $r$ carries. Kummer gives
\[
  v_p\binom{mq}{mb}\le r.
\]

Now suppose $b=0$. If $\lambda\ge2$, take $k=mp^t$. Removing the common factor $p^t$ from the addition reduces the carry count to that of
\[
  m+(\lambda-1)m=\lambda m<pm<p^{r+1},
\]
so again there are at most $r$ carries.

Finally, if $b=0$ and $\lambda=1$, then $q=p^t$ and $q\ge2$ implies $t\ge1$. Take $k=mp^{t-1}$. After a shift by $p^{t-1}$ the relevant addition is
\[
  m+(p-1)m=pm<p^{r+1},
\]
which again has at most $r$ carries.

In each case we have exhibited an admissible $k$ with $v_p\binom Nk\le r$. Equation~\eqref{eq:gcdmin} now yields $v_p(G(N;m))\le r$.
\end{proof}

The witness depends on the leading-digit decomposition of the row multiplier $q=N/m$. This dependence is essential: a fixed choice such as $k=m$ need not give the sharp row-independent bound.

\section{Constructive attainment}\label{sec:attainA}

We now prove that the upper bound is sharp. The two cases $r_p(m)=1$ and $r_p(m)\ge2$ require different constructions.

\begin{proposition}[Attainment when $r_p(m)=1$]\label{prop:r1}
Let $p$ be prime, $m\ge2$, $p\nmid m$, and $r_p(m)=1$. Then
\[
  v_p(G(mp;m))=1.
\]
\end{proposition}

\begin{proof}
The condition $r_p(m)=1$ means $m<p$. Every selected lower index in row $mp$ is $k=md$ with $1\le d\le p-1$. Since neither $m$ nor $d$ is divisible by $p$, the units digit of $md$ in base $p$ is nonzero. The same holds for $m(p-d)$, and
\[
  md+m(p-d)=mp
\]
is divisible by $p$. Hence the two units digits sum to $p$, producing a carry in the units column. Kummer therefore gives
\[
  v_p\binom{mp}{md}\ge1
\]
for every selected $d$. Thus $v_p(G(mp;m))\ge1$. Proposition~\ref{prop:upper} gives the reverse inequality.
\end{proof}

\begin{proposition}[Attainment when $r_p(m)\ge2$]\label{prop:rge2}
Let $p$ be prime, $m\ge2$, $p\nmid m$, and put $r=r_p(m)\ge2$. Write
\[
  P=p^{r-1},\qquad m=A P+c,
\]
where
\[
  1\le A\le p-1,\qquad 1\le c<P.
\]
Let $h=\operatorname{ord}_m(p)$, and choose $L$ such that
\[
  L\equiv r-1\pmod h,
  \qquad L\ge2r-2,
  \qquad L>r-1.
\]
Then the row
\[
  N=A p^L+c
\]
is admissible and satisfies
\[
  v_p(G(N;m))=r.
\]
\end{proposition}

\begin{proof}
Minimality of $r$ gives
\[
  P<m<pP.
\]
Because $p\nmid m$, the remainder $c$ in the displayed decomposition is nonzero, and $p$ has a multiplicative order modulo $m$. The choice of $L$ gives
\[
  p^L\equiv p^{r-1}=P\pmod m.
\]
Consequently
\[
  N=A p^L+c\equiv AP+c=m\equiv0\pmod m.
\]
Also $L>r-1$ implies $N>m$, so $N$ is admissible.

It remains to show that every selected binomial coefficient in this row has valuation at least $r$. Let
\[
  N=x+y,
  \qquad x>0,\ y>0,
  \qquad m\mid x,\ m\mid y.
\]
Suppose, toward a contradiction, that the base-$p$ addition $x+y$ has at most $r-1$ carries.

There must be a carry across the $p^L$ boundary. Otherwise write
\[
  x=A_1p^L+u,
  \qquad y=B_1p^L+v,
  \qquad 0\le u,v<p^L.
\]
No carry across that boundary gives
\[
  u+v=c,\qquad A_1+B_1=A.
\]
Since $m\mid x$ and $p^L\equiv P\pmod m$,
\[
  A_1P+u\equiv0\pmod m.
\]
But
\[
  0\le A_1P+u\le AP+c=m.
\]
Hence $A_1P+u$ is either $0$ or $m$. In the first case $x=0$; in the second, $A_1=A$ and $u=c$, forcing $y=0$. Both are impossible.

Therefore a carry crosses the $p^L$ boundary. Consider the final consecutive chain of carries ending there, and let its length be $\ell$. Our contradictory hypothesis gives
\[
  1\le\ell\le r-1.
\]
Set $s=L-\ell$. By maximality of this final chain, there is no carry across the $p^s$ boundary. Moreover
\[
  s\ge L-r+1\ge r-1,
\]
so $c<p^{r-1}\le p^s$. Write
\[
  x=A_2p^s+u,
  \qquad y=B_2p^s+v,
  \qquad 0\le u,v<p^s.
\]
The absence of a carry across $p^s$ gives
\[
  u+v=c,
  \qquad A_2+B_2=A p^\ell.
\]
Since $s+\ell=L$ and $p^L\equiv p^{r-1}\pmod m$, cancellation of the invertible factor $p^\ell$ modulo $m$ gives
\[
  p^s\equiv p^{r-1-\ell}\pmod m.
\]
Put $D=p^{r-1-\ell}$. From $m\mid x$ we obtain
\[
  A_2D+u\equiv0\pmod m.
\]
On the other hand,
\[
  0\le A_2D+u
  \le A p^\ell p^{r-1-\ell}+c
  =AP+c=m.
\]
Thus $A_2D+u$ is $0$ or $m$, and exactly as above either $x=0$ or $y=0$. This contradiction shows that every such split has at least $r$ carries.

Applying the argument to $x=k$ and $y=N-k$ for each selected $k$, Kummer yields
\[
  v_p\binom Nk\ge r.
\]
Hence $v_p(G(N;m))\ge r$, while Proposition~\ref{prop:upper} supplies the opposite inequality.
\end{proof}

Combining the upper bound and the two attainment cases gives the main maximum theorem.

\begin{theorem}[Sharp maximum]\label{thm:A}
Let $p$ be prime and let $m\ge2$ satisfy $p\nmid m$. Then
\[
  \max_{\substack{N>m\\m\mid N}}v_p(G(N;m))=r_p(m).
\]
Moreover, an equality row exists constructively: one may take $N=mp$ when $r_p(m)=1$, and when $r_p(m)\ge2$ one may take a row $N=A p^L+c$ as in Proposition~\ref{prop:rge2}.
\end{theorem}

\section{The least extremal row}\label{sec:Tdef}

Theorem~\ref{thm:A} contains an existence proof, so the following definition is now legitimate.

\begin{definition}\label{def:T}
For a prime $p$ and $m\ge2$ with $p\nmid m$, let
\[
  \mathcal E_{p,m}
  =\left\{N>m:m\mid N,\ v_p(G(N;m))=r_p(m)\right\}.
\]
By Theorem~\ref{thm:A}, $\mathcal E_{p,m}$ is nonempty. Define
\[
  T_p(m)=\min\mathcal E_{p,m}.
\]
\end{definition}

The order here is intentional: the least extremal row is not invoked until an attaining row has been constructed.

\section[Attainment for p-power plus one]{Attainment for $m=p^a+1$}\label{sec:Battain}

Fix a prime $p$ and an integer $a\ge2$. Put
\[
  Q=p^a,
  \qquad m=Q+1.
\]
Then
\[
  p^a<m<p^{a+1},
\]
so $r_p(m)=a+1$. Define
\[
  q_0=Q^2-Q+1,
  \qquad N_0=mq_0=Q^3+1=p^{3a}+1.
\]
The identity $q_0=\Phi_6(Q)$ is elementary and will only be used as a convenient description of the multiplier.

\begin{proposition}[Exact target-row valuation]\label{prop:Battain}
For every prime $p$ and every $a\ge2$,
\[
  v_p\bigl(G(p^{3a}+1;p^a+1)\bigr)=a+1.
\]
\end{proposition}

\begin{proof}
Let $1\le d<q_0$ and set $k=md$. Since
\[
  N_0=p^{3a}+1,
\]
we have
\[
  N_0\bmod p^i=1
  \qquad(1\le i\le3a).
\]
For $i\ge3a+1$ one has $p^i>N_0$, so there are no borrow levels beyond $3a$.

Suppose first that $p\mid k$. Since $p\nmid m$, put
\[
  t=v_p(k)=v_p(d)\ge1.
\]
For $i\le t$, $k\bmod p^i=0$, while for $i>t$ the residue $k\bmod p^i$ is a nonzero multiple of $p^t$, hence is greater than $1$. Lemma~\ref{lem:kummer-residue} therefore gives exactly $3a-t$ borrows among the levels $1,\dots,3a$. Since
\[
  d<q_0<Q^2=p^{2a},
\]
we have $t\le2a-1$, and hence $3a-t\ge a+1$.

Now suppose $p\nmid k$, and let
\[
  s=v_p(k-1).
\]
For $i\le s$, $k\bmod p^i=1$. For $i>s$, the residue is nonzero and is not $1$, hence is greater than $1$. Thus the borrow count is $3a-s$. We claim that $s\le2a-1$. If $s\ge2a$, then
\[
  md=k\equiv1\pmod{Q^2}.
\]
But
\[
  (1+Q)(1-Q)=1-Q^2\equiv1\pmod{Q^2},
\]
so
\[
  m^{-1}\equiv1-Q\equiv Q^2-Q+1=q_0\pmod{Q^2}.
\]
It follows that $d\equiv q_0\pmod{Q^2}$, impossible because $1\le d<q_0<Q^2$. Hence $s\le2a-1$, and again $3a-s\ge a+1$.

Thus every selected coefficient has valuation at least $a+1$. To obtain equality, take
\[
  d_0=p^{2a-1}.
\]
For $a\ge2$ one has $d_0<q_0$, so this is an admissible selected multiplier. Then
\[
  k_0=md_0=p^{3a-1}+p^{2a-1}
\]
has $v_p(k_0)=2a-1$. By the first case, the borrow levels are precisely
\[
  2a,2a+1,\dots,3a,
\]
a total of $a+1$. Therefore one selected coefficient has valuation exactly $a+1$, and the GCD has valuation $a+1$.
\end{proof}

\section{Strict non-attainment below the target row}\label{sec:Bminimal}

To prove minimality, let
\[
  2\le q<q_0
\]
and consider the admissible row $N=mq$. We will construct a multiplier $d$ with $1\le d<q$ such that
\begin{equation}\label{eq:strictgoal}
  v_p\binom{mq}{md}\le a.
\end{equation}
By \eqref{eq:gcdmin}, this proves that the row is not extremal.

Write the leading base-$p$ decomposition
\[
  q=\lambda p^t+b,
  \qquad1\le\lambda\le p-1,
  \qquad0\le b<p^t.
\]

\subsection{Pure leading term}

Assume first that $b=0$.

If $\lambda\ge2$, choose $d=p^t$. After removing the common shift $p^t$, the carries in
\[
  md+m(q-d)
\]
are the carries in
\[
  m+(\lambda-1)m=\lambda m.
\]
Since
\[
  \lambda m\le(p-1)(Q+1)<pQ=p^{a+1},
\]
there are at most $a$ carries, proving \eqref{eq:strictgoal}.

If $\lambda=1$, then $q=p^t$ and $t\ge1$. Choose $d=p^{t-1}$. After shifting, the addition is
\[
  m+(p-1)m=pm.
\]
The base-$p$ expansion of $m=Q+1$ has a $1$ in positions $0$ and $a$ and zeroes elsewhere. Hence this addition produces exactly two carries, one from position $0$ and one from position $a$. Since $a\ge2$, the valuation is at most $a$.

\subsection{Nonzero suffix and the exceptional leading pattern}

Now suppose $b>0$. We first try the leading-split witness $d=b$. Put
\[
  X=\left\lfloor\frac{mb}{p^t}\right\rfloor.
\]
Because $b<p^t$,
\[
  0\le X<m=Q+1,
\]
so $0\le X\le Q$. As in the proof of Proposition~\ref{prop:upper}, the carry count is the carry count in
\[
  X+\lambda m.
\]
This count is at most $a+1$. If it is at most $a$, then \eqref{eq:strictgoal} is proved.

Suppose instead that there are $a+1$ carries. Then all possible carry positions $0,1,\dots,a$ occur. Since
\[
  \lambda m=\lambda Q+\lambda,
\]
this forces
\begin{equation}\label{eq:exceptionalpattern}
  \lambda=p-1,
  \qquad Q-p+1\le X\le Q-1.
\end{equation}
Indeed, carrying through positions $0,\dots,a-1$ forces the intermediate digits of $X$ to be $p-1$ and its units digit to be nonzero, while the carry out of position $a$ forces $\lambda=p-1$.

Set
\[
  h=p^t-b.
\]
Then, using \eqref{eq:exceptionalpattern},
\[
  q=p^{t+1}-h
\]
and
\[
  X=Q+1-\left\lceil\frac{mh}{p^t}\right\rceil.
\]
The exceptional range is equivalent to
\[
  2\le\left\lceil\frac{mh}{p^t}\right\rceil\le p,
\]
which, because $m$ is coprime to $p$, yields
\begin{equation}\label{eq:mhwindow}
  p^t<mh<p^{t+1}.
\end{equation}
The upper inequality implies $t\ge a$: if $t\le a-1$, then $p^{t+1}\le Q<m\le mh$, a contradiction. Since $q<q_0<Q^2$, also $t\le2a-1$. The endpoint $t=2a-1$ cannot occur: from $q<q_0$ one would obtain $h\ge Q$, and hence $mh>Q^2=p^{t+1}$, contradicting \eqref{eq:mhwindow}. Therefore
\begin{equation}\label{eq:trange}
  a\le t\le2a-2.
\end{equation}
Moreover \eqref{eq:mhwindow} gives
\[
  h<\frac{p^{t+1}}m<p^{t-a+1}\le p^{a-1}.
\]

\subsection{The fallback witness}

Put
\[
  P=p^{a-1}
\]
and choose the fixed fallback multiplier
\[
  d=P.
\]
Because $q>p^t\ge Q>P$, this satisfies $1\le d<q$.

The selected lower index is
\[
  K=mP=P+p^{2a-1}.
\]
Using $q=p^{t+1}-h$, write
\[
  N=mq=p^{t+a+1}+R,
  \qquad R=p^{t+1}-mh>0,
\]
and put $L=t+a+1$.

We show first that no borrow occurs at any level $i<2a$. If $i\le a-1$, then
\[
  K\bmod p^i=0.
\]
If $a\le i\le2a-1$, then
\[
  K\bmod p^i=P,
\]
so it remains to prove $N\bmod p^i>P$.

Suppose $i\le t+1$. Then
\[
  R\bmod p^i=(-mh)\bmod p^i.
\]
Since $h<P$, the summands $h$ and $Qh$ in $mh=h+Qh$ occupy disjoint base-$p$ digit blocks. Thus
\[
  mh\bmod p^i
  =h+Q\bigl(h\bmod p^{i-a}\bigr)
\]
and consequently
\begin{align*}
  mh\bmod p^i
  &\le(P-1)+Q(p^{i-a}-1)\\
  &=p^i-(Q-P+1)\\
  &\le p^i-P-1.
\end{align*}
Hence
\[
  R\bmod p^i\ge P+1.
\]

Now suppose $i>t+1$. Then $R<p^{t+1}<p^i$, so $N\bmod p^i=R$. By \eqref{eq:trange},
\[
  s=t+1-a
\]
satisfies $1\le s\le a-1$. Modulo $m=Q+1$,
\[
  p^{t+1}=Qp^s\equiv-p^s\pmod m.
\]
Therefore the least positive residue of $p^{t+1}$ modulo $m$ is
\[
  m-p^s=Q+1-p^s\ge Q+1-P\ge P+1.
\]
Since $R>0$ and $R\equiv p^{t+1}\pmod m$, it follows that $R\ge P+1$.

We have proved that no borrow occurs below level $2a$. At the other end, if $i\ge L+1$, then $p^i>N$, so no borrow occurs there either. All borrows are therefore confined to
\[
  2a\le i\le L.
\]
The number of such levels is at most
\[
  L-2a+1=t-a+2\le a
\]
by \eqref{eq:trange}. Kummer gives \eqref{eq:strictgoal}.

We have now covered every $q$ below $q_0$.

\begin{proposition}[Strict lower-row non-attainment]\label{prop:lowernon}
Let $p$ be prime and $a\ge2$. If $N$ is admissible for $p^a+1$ and
\[
  N<p^{3a}+1,
\]
then
\[
  v_p\bigl(G(N;p^a+1)\bigr)\le a.
\]
\end{proposition}

\begin{proof}
Write $N=(p^a+1)q$. Admissibility gives $q\ge2$, and the factorization
\[
  p^{3a}+1=(p^a+1)(p^{2a}-p^a+1)
\]
shows that $N<p^{3a}+1$ is equivalent to $q<q_0$. The construction above produces $1\le d<q$ satisfying \eqref{eq:strictgoal}, so the selected-GCD valuation is at most $a$.
\end{proof}

Combining Propositions~\ref{prop:Battain} and \ref{prop:lowernon} gives the least-row formula.

\begin{theorem}[Least extremal row for $p^a+1$]\label{thm:B}
For every prime $p$ and every integer $a\ge2$,
\[
  \boxed{T_p(p^a+1)=p^{3a}+1.}
\]
Equivalently, the least multiplier $q\ge2$ such that
\[
  v_p(G((p^a+1)q;p^a+1))=a+1
\]
is
\[
  q=p^{2a}-p^a+1=\Phi_6(p^a).
\]
\end{theorem}

\section{Examples and boundary cases}\label{sec:examples}

\begin{example}[$T_2(5)=65$]
Take $p=2$ and $a=2$. Then $m=2^2+1=5$ and $r_2(5)=3$. Theorem~\ref{thm:B} gives
\[
  T_2(5)=2^6+1=65.
\]
The corresponding least multiplier is
\[
  q_0=2^4-2^2+1=13,
\]
so $65=5\cdot13$. At the target row, the equality witness from Proposition~\ref{prop:Battain} uses
\[
  d_0=2^{2a-1}=8,
  \qquad k_0=5\cdot8=40,
\]
and $v_2\binom{65}{40}=3$.
\end{example}

\begin{example}[An $r_p(m)=1$ equality row]
Let $p=5$ and $m=3$. Then $r_5(3)=1$, and Proposition~\ref{prop:r1} gives
\[
  v_5(G(15;3))=1.
\]
The argument does not require the congruence $5\equiv1\pmod3$.
\end{example}

\begin{remark}[Why $a\ge2$]
The proof of Theorem~\ref{thm:B} uses $a\ge2$ essentially. In the pure-power multiplier branch, one witness can have exactly two carries, which is bounded by $a$ only from $a=2$ onward. The $a=1$ regime is therefore not included in the theorem.
\end{remark}

\begin{remark}[The known $(2,3,6)$ instance]
McTague explicitly records
\[
  v_2(G(6;3))=2
\]
in the same multiple-index family \cite{McTague2017}. Since $6$ is the least admissible row for $m=3$, this already yields $T_2(3)=6$. This instance is prior art and is not included in Theorem~\ref{thm:B}.
\end{remark}

\section{Related work and literature position}\label{sec:related}

Ram's classical full-row GCD theorem initiated a long line of questions about common divisors of binomial coefficients \cite{Ram1909}; Kummer's theorem supplies the standard carry interpretation of their prime valuations \cite{Kummer1852}. Granville gives a modern treatment of arithmetic properties of binomial coefficients modulo prime powers \cite{Granville1997}.

For the multiple-index restriction relevant here, McTague proved a theorem for
\[
  \gcd_{0<k<n/q}\binom n{qk}
\]
when the prime satisfies $p\equiv1\pmod q$ \cite{McTague2017}. In the notation of this paper, that result contains the entire subcase $p\equiv1\pmod m$ of Theorem~\ref{thm:A}: then $p>m$, so $r_p(m)=1$, and McTague's criterion gives the corresponding $0/1$ valuation behavior and an equality row. McTague also gives the explicit $(p,m,N)=(2,3,6)$ valuation-$2$ example discussed above.

Wu subsequently defined
\[
  g(m,n)=\gcd\left\{\binom{mn}{mk}:1\le k<n\right\},
\]
which is exactly $G(mn;m)$, and proved further results for this family, including the case in which the restriction parameter itself is a power of the valuation prime \cite{Wu2026}. Thus neither the selected-GCD family nor the use of Kummer carries is being introduced here.

Several other papers study GCDs under different restrictions on the lower indices. Hong treats indices coprime to a fixed parameter \cite{Hong2016}; Chiu, Yuan and Zhou study central-band GCDs separated by coprimality with the row \cite{ChiuYuanZhou2023}; and Chung, Yang and Zhou study exact GCD classes $\gcd(n,k)=d$ \cite{ChungYangZhou2025}. These restrictions are related in theme but are not the same as the fixed divisibility condition $m\mid k$ used here.

A particularly recent adjacent paper is Chung and Yang's 2026 work on binomial coefficients with non-coprime indices \cite{ChungYang2026}. During the documented literature audit for the present work, the publisher abstract and metadata were inspectable, but the full theorem text was not openly available. The advertised selection rule is different: it ranges over lower indices non-coprime to the upper index, whereas the present paper fixes a divisor $m$ and selects the indices divisible by $m$. The complete theorem-level non-overlap with material inaccessible in that paper could therefore not be verified.

The documented literature search located no equivalent statement of the general maximum in Theorem~\ref{thm:A} outside the known overlaps just described, and no equivalent statement of Theorem~\ref{thm:B}, including its least multiplier and strict lower-row non-attainment. This is a negative search result only; it is not a claim of historical uniqueness.

\section{Formal verification}\label{sec:formal}

The mathematical proofs above stand independently of the accompanying formal development. They have also been formalized in Lean~4 against Mathlib. The principal theorem layer is contained in the \texttt{PascalExtremes} directory: \texttt{TargetA.lean} proves the universal bound; \texttt{TargetAAttainment.lean} proves constructive attainment and least-row well-definedness; \texttt{TargetB.lean} proves target-row attainment; and \texttt{TargetBMinimality.lean} proves strict lower-row non-attainment and the least-row theorem. In particular, the formal theorem \texttt{targetA\_attainment} precedes the use of the least extremal row, \texttt{targetA} states the exact maximum as an \texttt{IsGreatest} assertion, and \texttt{targetB} states
\[
  T\,p\,(p^a+1)=p^{3a}+1
\]
for every prime $p$ and $a\ge2$.

A Mathlib-only statement surface and its solution bridge were registered publicly with Palomar as \textbf{PALOMAR-2026-09-30-000033}, version~1; the \href{https://palomar-registry.org/entry.html?id=PALOMAR-2026-09-30-000033\&version=1}{public registry entry} records that version. The registration records the verified Lean theorem package and its protected statement surface. It is verification and provenance evidence; it is not evidence of historical priority, novelty, or uniqueness.

\section{Concluding remarks}

The sharp maximum in Theorem~\ref{thm:A} has a simple numerical value, but its attainment is not supplied by a uniform row shape such as $mp^r$. The successful construction instead repeats the leading base-$p$ block of $m$ at a sufficiently distant exponent constrained by the multiplicative order of $p$ modulo $m$. For $m=p^a+1$, the extremal problem is much more rigid: Theorem~\ref{thm:B} identifies not merely an equality row but the least one, and the proof of minimality requires recognizing and repairing the unique leading-digit pattern in which the first witness loses the desired strictness.

\bibliographystyle{alpha}
{\small
\bibliography{references}
}

\end{document}